\documentclass[a4paper]{amsart}
\usepackage{graphicx}
\usepackage{amsmath}
\usepackage{amssymb}
\usepackage{oldgerm}
\usepackage[all]{xy}
\usepackage{color}

\newcommand{\Hom}{\operatorname{Hom}\nolimits}

\newcommand{\End}{\operatorname{End}\nolimits}

\renewcommand{\mod}{\operatorname{mod}\nolimits}
\newcommand{\stmod}{\operatorname{\underline{mod}}\nolimits}

\newcommand{\gldim}{\operatorname{gldim}\nolimits}
\newcommand{\repdim}{\operatorname{repdim}\nolimits}

\newcommand{\Ext}{\operatorname{Ext}\nolimits}

\newcommand{\HH}{\operatorname{HH}\nolimits}

\newcommand{\ra}{\operatorname{\mathfrak{r}}\nolimits}
\newcommand{\La}{\Lambda}

\newcommand{\op}{\operatorname{op}\nolimits}

\newcommand{\e}{\operatorname{e}\nolimits}
\newcommand{\Lae}{\operatorname{\Lambda^{\e}}\nolimits}

\newcommand{\add}{\operatorname{add}\nolimits}

\newcommand{\Aut}{\operatorname{Aut}\nolimits}

\newtheorem{theorem}{Theorem}[section]

\newtheorem{corollary}[theorem]{Corollary}

\newtheorem{lemma}[theorem]{Lemma}

\theoremstyle{definition}

\theoremstyle{definition}

\theoremstyle{definition}

\theoremstyle{definition}
\newtheorem*{example}{Example}
\theoremstyle{definition}

\theoremstyle{definition}

\theoremstyle{remark}

\theoremstyle{remark}

\theoremstyle{definition}

\theoremstyle{definition}

\begin{document}
\title[Skew group rings, wreath products and blocks of Hecke algebras]{On the representation dimension of skew group algebras, wreath products and blocks of Hecke algebras}
\author{Petter Andreas Bergh \& Will Turner}
\address{Petter Andreas Bergh \newline Institutt for matematiske fag \\
  NTNU \\ N-7491 Trondheim \\ Norway}
\email{bergh@math.ntnu.no}
\address{Will Turner \newline Department of Mathematics \\ University of Aberdeen \\ Fraser Noble Building \\ King's College \\ Aberdeen AB24 3UE \\ United Kingdom}
\email{w.turner@abdn.ac.uk}




\thanks{This paper grew out of a visit made by the first author to the Department of Mathematics, University of Aberdeen, in December 2010. He would like to thank the algebra group of the department for the invitation and the hospitality.}

\begin{abstract}
We establish bounds for the representation dimension of skew group algebras and wreath products. Using this, we obtain bounds for the representation dimension of a block of a Hecke algebra of type $A$, in terms of the weight of the block. This includes certain blocks of group algebras of symmetric groups.
\end{abstract}

\maketitle

\section{Introduction}\label{secintro}

In \cite{Auslander1}, Auslander introduced the representation dimension of a finite dimensional algebra. An algebra is semisimple precisely when its representation dimension is zero, and no algebra has representation dimension one. A non-semisimple algebra is of finite representation type precisely when its representation dimension is two, and of infinite type when the representation dimension is at least three. The motivation for introducing this invariant was to measure how far an algebra is from having finite representation type. However, the full meaning and the properties of this invariant are far from understood, except in the very lowest dimensions. For example, a still unproved conjecture states that the representation dimension of an algebra of tame representation type is at most three. 

Thirtytwo years after the introduction of this invariant, Iyama showed in \cite{Iyama} that it is always finite. Three years later, the first examples appeared of algebras having representation dimension greater than three. Namely, Rouquier showed in \cite{Rouquier1} that the representation dimension of the exterior algebra on a $d$-dimensional vector space is $d+1$. Other examples, giving lower bounds, followed in  
\cite{Bergh}, \cite{BerghOppermann1}, \cite{BerghOppermann2}, \cite{KrauseKussin}, \cite{Oppermann1}, \cite{Oppermann2}, \cite{Oppermann3}, \cite{OppermannMiemietz}, using Rouquier's notion of the dimension of a triangulated category (cf.\ \cite{Rouquier2}). However, the exact value of the representation dimension is known in only a few cases, and there does not even exist a method for computing an effective \emph{upper} bound. 

In \cite{BerghErdmann}, both an upper and a lower bound for the representation dimension of Hecke algebras were established. Let $\mathcal{H}_q(A_{n-1})$ be such an algebra of type $A_{n-1}$, where $q$ is a primitive $\ell^{th}$ root of unity and the ground field $k$ is of characteristic zero. If the algebra is not semisimple, then it was shown that
$$[n/ \ell ] +1 \le \repdim \mathcal{H}_q(A_{n-1}) \le 2 [n/ \ell ],$$
where $[n/ \ell]$ is the integer part of the rational number $n/ \ell$. The same lower bound was established for Hecke algebras of types $B$ and $D$, and the upper bound was shown to hold in certain cases.

In this paper, we improve this result by extending it to \emph{blocks} of Hecke algebras of type $A$. 
We show that if $\mathcal{B}$ is a block of $\mathcal{H}_q(A_{n-1})$ of weight $w$, then
$$w+1 \le \repdim \mathcal{B} \le 2w.$$
The weight of any block is at most $[n/ \ell]$, and the weight of the principal block equals $[n/ \ell]$. 
Therefore the bounds from \cite{BerghErdmann} follow easily from our result.
The same bounds also hold for blocks of $\mathcal{H}_q(A_{n-1})$ of weight $w$ over fields $k$ of characteristic $p>w$, 
in case $q \in \mathbb{F}_p \backslash \{0,1 \}$ is an $\ell^{th}$ root of unity; 
and they hold in case $q=1$ and $\ell =p$ where the Hecke algebra $\mathcal{H}_q(A_{n-1})$ 
is nothing but the group algebra $kS_n$ of the $n^{th}$ symmetric group. 
 
Our proof relies on results on so-called Rouquier blocks of $\mathcal{H}_q(A_{n-1})$, cf.\  \cite{ChuangKessar} and \cite{ChuangMiyachi}. 
Such a block of weight $w$ is Morita equivalent to the wreath product $A \wr S_w$, 
where $A$ is a Brauer tree algebra associated to the graph
$$\xymatrix{
\circ \ar@{-}[r] & \circ \ar@{-}[r] & \cdots \ar@{-}[r] & \circ \ar@{-}[r] & \circ }$$
with $\ell$ vertices and no exceptional multiplicity, and $S_w$ is the $w^{th}$ symmetric group. 
The representation dimension of the block then equals that of the algebra $A \wr S_w$. We therefore establish bounds on the representation dimension of wreath products, and of more general skew group algebras.

\section{Skew group algebras and wreath products}\label{sec skew}

In this section, we fix a field $k$. Let $\Lambda$ be a finite dimensional $k$-algebra,
and denote by $\mod \Lambda$ the category of finitely generated
left $\Lambda$-modules. The
\emph{representation dimension} of $\Lambda$, denoted $\repdim
\Lambda$, is defined as
$$\repdim \Lambda \stackrel{\text{def}}{=} \inf \{ \gldim
\End_{\Lambda}(M) \mid M \text{ generates and cogenerates } \mod
\Lambda \},$$ where $\gldim$ denotes the global dimension of an
algebra. To say that a module generates and cogenerates $\mod
\Lambda$ means that it contains all the indecomposable projective
and injective modules as direct summands. Of course, when $\La$ is selfinjective, a module is a generator if and only if it is a cogenerator.

We shall prove bounds for the representation dimension of blocks of Hecke algebras. As mentioned in the introduction, we do this by using a derived equivalence from \cite{ChuangMiyachi}, which enables us to pass to certain wreath products. Wreath products are special skew group algebras, and we therefore start with computing bounds for such algebras.

Suppose $G$ is a finite group acting on $\La$, i.e.\ there is a group homomorphism $G \to \Aut \La$, where $\Aut \La$ is the multiplicative group of algebra automorphisms on $\La$. The \emph{skew group algebra} $\La [G]$ is the finite dimensional algebra whose underlying vector space is $\La \otimes_k kG$, where $kG$ is the group algebra of $G$. The multiplication is defined by
$$( \lambda \otimes g )( \lambda' \otimes g' ) \stackrel{\text{def}}{=} \lambda g( \lambda' ) \otimes gg',$$
for $\lambda, \lambda' \in \La$ and $g,g' \in G$. Denote the identity in $G$ by $e$. The algebra homomorphism
\begin{eqnarray*}
\La & \to & \La [G] \\
\lambda & \mapsto & \lambda \otimes e
\end{eqnarray*}
is injective, and so we may view $\La$ as a subalgebra of $\La [G]$. Since every element in the skew group algebra can be written uniquely as a sum $\sum_{g \in G} \lambda_g \otimes g$, there is an isomorphism
$$\La [G] \simeq \bigoplus_{g \in G} {{_1\La}_g}$$
of $\La$-$\La$-bimodules. Here, the bimodule ${{_1\La}_g}$ is the vector space $\La$, with bimodule scalar action given by $\lambda_1 \cdot \lambda \cdot \lambda_2 = \lambda_1 \lambda  g( \lambda_2 )$. In particular, the skew group algebra is free as a left $\La$-module.

In the above bimodule isomorphism, the summand ${{_1\La}_e}$ is isomorphic (actually equal) to $\La$. Therefore $\La$ is isomorphic to a direct summand of $\La [G]$, as $\La$-$\La$-bimodules. Now let $X$ be $\La [G]$ considered as a $\La$-$\La [G]$-bimodule, and let $Y$ be $\La [G]$ considered as a $\La [G]$-$\La$-bimodule. From what we have just seen, the $\La$-$\La$-bimodule $\La$ is isomorphic to a direct summand of $X \otimes_{\La [G]} Y$. By \cite[Theorem 1.1]{ReitenRiedtmann}, if the order of $G$ does not divide the characteristic of the ground field $k$, then the ``opposite" also holds. Namely, in this situation, the $\La [G]$-$\La[G]$-bimodule $\La [G]$ is isomorphic to a direct summand of $Y \otimes_{\La} X$. In the terminology of \cite{Linckelmann}, the algebras $\La$ and $\La [G]$ are then \emph{separably equivalent}. Moreover, by \cite[Theorem 1.1 and Theorem 1.3(c)]{ReitenRiedtmann}, one of the algebras is selfinjective if and only if the other is. We record all this in the following lemma.

\begin{lemma}\label{SkewGroupAlg}
Let $\La$ be a finite dimensional algebra on which a finite group $G$ is acting, and suppose that the order of $G$ is invertible in the ground field. Then the following hold:
\begin{enumerate}
\item $\La$ and $\La[G]$ are separably equivalent.
\item $\La$ is selfinjective if and only if $\La[G]$ is.
\end{enumerate}
\end{lemma}

When $\La$ and $\La[G]$ are selfinjective, their stable module categories $\stmod \La$ and $\stmod \La[G]$ are triangulated categories. When these algebras are separably equivalent, the dimensions of these triangulated categories (in the sense of \cite{Rouquier2}), are equal by \cite[Corollary 3.7]{Linckelmann} (the latter is stated for symmetric algebras, but holds for separably equivalent selfinjective algebras). Combining this with Lemma \ref{SkewGroupAlg}, we obtain the following lemma.

\begin{lemma}\label{SkewGroupAlgDim}
Let $\La$ be a finite dimensional selfinjective algebra on which a finite group $G$ is acting, and suppose that the order of $G$ is invertible in the ground field. Then 
$$\dim \left ( \stmod \La \right ) = \dim \left ( \stmod \La [G] \right ).$$
\end{lemma}

We can now give a lower bound for the representation dimension of a skew group algebra.

\begin{theorem}\label{SkewGroupAlgLowerBound}
Let $\La$ be a finite dimensional selfinjective algebra on which a finite group $G$ is acting, and suppose that the order of $G$ is invertible in the ground field. Then 
$$\dim \left ( \stmod \La \right ) +2 \le \repdim \La [G].$$
\end{theorem}

\begin{proof}
This follows from Lemma \ref{SkewGroupAlgDim} and \cite[Proposition 3.7]{Rouquier1}.
\end{proof}

Obtaining an upper bound for the representation dimension of a skew group algebra is more complicated. The bound we give is in terms of the restriction of an induced module. To be precise, let $M$ be a left $\La$-module. Then the induced module $\La [G] \otimes_{\La} M$ is a finitely generated left $\La [G]$-module. What is the restriction of this module to $\La$? Recall that $\La [G]$ and $\oplus_{g \in G} {{_1\La}_g}$ are isomorphic as $\La$-$\La$-bimodules, hence the left $\La$-module $\La [G] \otimes_{\La} M$ is isomorphic to $\oplus_{g \in G} \left ( {{_1\La}_g} \otimes_{\La} M \right )$. Given any element $g \in G$, the left $\La$-modules ${{_1\La}_g} \otimes_{\La} M$ and ${_{g^{-1}}M}$ are isomorphic, where ${_{g^{-1}}M}$ is just $M$ with scalar multiplication $\lambda \cdot m = g^{-1}( \lambda )m$. Since the direct sum runs over all elements in $G$, we have shown that the restriction of $\La [G] \otimes_{\La} M$ to $\La$ is isomorphic to $\oplus_{g \in G} {{_g}M}$. The following result shows that when $M$ is a generator in $\mod \La$, and $\oplus_{g \in G} {{_g}M}$ belongs to $\add_{\La} M$, then the global dimension of the endomorphism ring of $M$ is an upper bound for $\repdim \La [G]$.

\begin{theorem}\label{SkewGroupAlgUpperBound}
Let $\La$ be a finite dimensional selfinjective algebra on which a finite group $G$ is acting, and suppose that the order of $G$ is invertible in the ground field. Furthermore, let $M$ be a generator in $\mod \La$, and suppose that $\oplus_{g \in G} {{_g}M}$ belongs to $\add_{\La} M$. Then
$$\repdim \La [G] \le \gldim \End_{\La} (M).$$
\end{theorem}

\begin{proof}
Consider the left $\La [G]$-module $\La [G] \otimes_{\La} M$. By viewing $\La [G]$ as a $\La [G]$-$\La$-bimodule, the vector space $\Hom_{\La [G]}( \La [G], \La [G] \otimes_{\La} M)$ becomes a left $\La$-module, and as such it is isomorphic to the restriction of $\La [G] \otimes_{\La} M$ to $\La$. We have seen that this $\La$-module is isomorphic to $\oplus_{g \in G} {{_g}M}$, hence it belongs to $\add_{\La} M$ by assumption. Consequently, the left $\La$-module $\Hom_{\La [G]}( \La [G], \La [G] \otimes_{\La} M)$ belongs to $\add_{\La} M$. By \cite[Theorem 2.3]{BerghErdmann}, the inequality
$$\gldim \End_{\La [G]}( \La [G] \otimes_{\La} M) \le \gldim \End_{\La}(M)$$
then holds. Now, since the $\La [G]$-module $\La [G] \otimes_{\La} M$ is a generator by \cite[Proposition 2.4]{BerghErdmann}, the representation dimension of $\La [G]$ is at most the global dimension of $\End_{\La [G]}( \La [G] \otimes_{\La} M)$.
\end{proof}

We give an example illustrating Theorem \ref{SkewGroupAlgLowerBound} and Theorem \ref{SkewGroupAlgUpperBound}.

\begin{example}
Suppose $\La$ is a selfinjective algebra of finite representation type on which a finite group $G$ is acting, and that the order of $G$ is invertible in the ground field. Let $M$ be the direct sum of a complete set of representatives of the isomorphism classes of the indecomposable $\La$-modules. Then $M$ is a generator, and 
$$\repdim \La = \gldim \End_{\La} (M) \le 2$$
provided $\La$ is not semisimple. If $g$ is any element in $G$, then ${{_g}M}$ is just isomorphic to $M$, since an algebra automorphism will just permute the indecomposable modules. Therefore Theorem \ref{SkewGroupAlgUpperBound} shows that the representation dimension of $\La [G]$ is also at most two, that is, the skew group algebra is also of finite representation type. This is contained in \cite[Theorem 1.3(a)]{ReitenRiedtmann}.
\end{example}

We now turn to a special class of skew group algebras. Let $n$ be a positive integer, and denote by $\La^{\otimes n}$ the tensor product of $n$ copies of our $k$-algebra $\La$. This is again a finite dimensional algebra, on which there is a natural action by the $n^{th}$ symmetric group $S_n$. Namely, for an element $\sigma \in S_n$ and $a_1 \otimes \cdots \otimes a_n \in \La^{\otimes n}$, the action is given by
$$\sigma ( a_1 \otimes \cdots \otimes a_n ) \stackrel{\text{def}}{=} a_{\sigma^{-1}(1)} \otimes \cdots \otimes a_{\sigma^{-1}(n)}.$$
The \emph{wreath product} $\Lambda \wr S_n$ is the skew group algebra $\La^{\otimes n} [S_n]$. In other words, it is the finite dimensional $k$-algebra whose underlying vector space is 
$$\La^{\otimes n} \otimes kS_n,$$ 
and with multiplication defined by
$$(a_1 \otimes \cdots \otimes a_n \otimes \sigma )(b_1 \otimes \cdots \otimes b_n \otimes \tau ) \stackrel{\text{def}}{=} a_1b_{\sigma^{-1}(1)} \otimes \cdots \otimes a_nb_{\sigma^{-1}(n)} \otimes \sigma \tau,$$
for $a_i,b_i \in \La$ and $\sigma, \tau \in S_n$. 

When $\La$ is selfinjective, then so is the tensor algebra $\La^{\otimes n}$. Using Theorem \ref{SkewGroupAlgLowerBound} and Theorem \ref{SkewGroupAlgUpperBound}, we obtain the following result, which gives lower and upper bounds on the representation dimension of wreath products. The proof of the upper bound requires the ground field to be perfect.

\begin{theorem}\label{WreathProduct}
Let $\La$ be a finite dimensional selfinjective non-semisimple algebra over a perfect field, 
and $n$ a positive integer with $n!$ invertible in the ground field. Then
$$\dim \left ( \stmod \La^{\otimes n} \right ) +2 \le \repdim \left ( \Lambda \wr S_n \right ) \le n \left ( \repdim \La \right ).$$
\end{theorem}

\begin{proof}
The order of the symmetric group $S_n$ is $n!$, which by assumption is invertible in the ground field. The lower bound
$$\dim \left ( \stmod \La^{\otimes n} \right ) +2 \le \repdim \left ( \Lambda \wr S_n \right )$$
therefore follows from Theorem \ref{SkewGroupAlgLowerBound}.

Let $d$ be the representation dimension of $\La$. Then there exists a generator $M \in \mod \La$ with $\gldim \End_{\La} (M) =d$. Let $t$ be an integer with the property that $\La$ is a direct summand of $M^t$, where $M^t$ is the direct sum of $t$ copies of $M$, and denote this direct sum by $N$. Then $N$ is also a generator in $\mod \La$ realizing the representation dimension of $\La$, i.e.\ $\gldim \End_{\La} (N) =d$. The tensor product $N^{\otimes n}$ is now a generator in $\mod \La^{\otimes n}$, since $\La^{\otimes n}$ is a direct summand. By \cite[Corollary 3.3 and Lemma 3.4]{Xi}, the global dimension of its endomorphism ring is given by
$$\gldim \End_{\La^{\otimes n}} \left ( N^{\otimes n} \right ) = \gldim \left ( \End_{\La}(N) \right )^{\otimes n} = n \left ( \gldim \End_{\La} (N) \right ) =nd.$$
For any element $\sigma \in S_n$, the twisted $\La^{\otimes n}$-module ${_{\sigma}(N^{\otimes n})}$ is isomorphic to $N^{\otimes n}$ itself via the isomorphism
\begin{eqnarray*}
N^{\otimes n} & \to & {_{\sigma}(N^{\otimes n})} \\
x_1 \otimes \cdots \otimes x_n & \mapsto & x_{\sigma^{-1}(1)} \otimes \cdots \otimes x_{\sigma^{-1}(n)}.
\end{eqnarray*}
Therefore, the $\La^{\otimes n}$-module $\oplus_{\sigma \in S_n} {_{\sigma}(N^{\otimes n})}$ trivially belongs to $\add_{\La^{\otimes n}} N^{\otimes n}$.
Consequently, the inequality
$$\repdim \left ( \Lambda \wr S_n \right ) \le \gldim \End_{\La^{\otimes n}} \left ( N^{\otimes n} \right ) =nd$$
follows from Theorem \ref{SkewGroupAlgUpperBound}.
\end{proof}

\section{Blocks of Hecke algebras}\label{secblocks}

As in the previous section, in this section we fix a field $k$. Let $q \in
k^\times$. The corresponding Hecke algebra $\mathcal{H}_q(A_{n-1})$ of type $A_{n-1}$ is the $k$-algebra with generators $T_1, \dots,
T_{n-1}$ satisfying the relations
\begin{eqnarray*}
(T_i+1)(T_i-q)=0 & \text{for} & 1 \le i \le n-1 \\
T_iT_{i+1}T_i = T_{i+1}T_iT_{i+1} & \text{for} & 1 \le i \le n-2 \\
T_iT_j = T_jT_i & \text{for} & |i-j| \ge 2.
\end{eqnarray*}
If $q=1$, this is just the group algebra of the symmetric
group $S_n$, hence $\mathcal{H}_q(A_{n-1})$ is also referred to as the Hecke algebra of $S_n$. In any case it is a symmetric algebra.

Let us assume that $\mathcal{H}_q(A_{n-1})$ is not semisimple for all $n$.
Then either $q \in k \backslash \{0,1 \}$, and the multiplicative order of $q$ in $k^\times$ is a positive integer $\ell$, 
or $q= 1$ in which case $k$ has positive characteristic $\ell >0$.
Under these assumptions the representation type of $\mathcal{H}_q(A_{n-1})$ depends on the number $[n/ \ell ]$, 
where $[r]$ is the integer part of a rational number $r$ (cf.\ \cite{ErdmannNakano}). Namely, the algebra $\mathcal{H}_q(A_{n-1})$ is semisimple if and only if $[n/ \ell ]=0$, where we define $[n/ \ell ]$ to be zero when $\ell = \infty$. It is non-semisimple of finite representation type if and only if $[n/ \ell ]=1$, and of tame representation type if and only if $\ell =2$ and $n$ is either $4$ or $5$ (and then $[n/ \ell ]=2$). In all other cases, the algebra $\mathcal{H}_q(A_{n-1})$ is of wild representation type. In \cite{BerghErdmann} it was shown that when the characteristic of the ground field $k$ is zero, then
$$[n/ \ell ] +1 \le \repdim \mathcal{H}_q(A_{n-1}) \le 2[n/ \ell ]$$
whenever $\mathcal{H}_q(A_{n-1})$ is not semisimple (that is, when $[n/ \ell ]>0$). We will extend this to blocks, where the weight plays the role of the integer $[n/ \ell ]$. 

The representation theory of the blocks of $\mathcal{H}_q(A_{n-1})$ is intimately connected with the theory of partitions of $n$; for details, see \cite{DipperJames}. Let $\lambda = ( \lambda_1, \dots, \lambda_t )$ be a partition of $n$, i.e.\ $\lambda_i \ge \lambda_{i+1}$ and $| \lambda | = \lambda_1 + \cdots + \lambda_t =n$. When we remove as many rim $\ell$-hooks as possible from the Young diagram of $\lambda$, we obtain (the Young diagram of) the \emph{$\ell$-core} of $\lambda$. Now, for each partition $\lambda$ of $n$ there is a $q$-Specht module $S^{\lambda}$ for  $\mathcal{H}_q(A_{n-1})$, and two such Specht modules $S^{\lambda}$ and $S^{\mu}$ belong to the same block precisely when the $\ell$-cores of the partitions $\lambda$ and $\mu$ are equal. Let $\mathcal{B}$ be the block containing the q-Specht module $S^{\lambda}$. Dividing $n$ by $\ell$ gives
$$n = w \ell + | \gamma |,$$
where $\gamma$ is the $\ell$-core of $\lambda$. The integer $w$ is then the \emph{weight} of $\mathcal{B}$. The block is called a \emph{Rouquier block} if its $\ell$-core $\gamma$ has at least $w-1$ more beads on each runner than on the runner to its left in an abacus presentation (cf.\ \cite{ChuangKessar}).

It follows from the definition that the weight $w$ of a block of $\mathcal{H}_q(A_{n-1})$ satisfies $0 \le w \le [n/ \ell ]$. Moreover, for every $w \in \{ 0, \dots, [n/ \ell ] \}$ there is a block of $\mathcal{H}_q(A_{n-1})$ of weight $w$. For example, the principal block, that is, the block containing the trivial module $k$, has weight $[n/ \ell ]$. To see this, write $n = m \ell + a$, where $0 \le a \le \ell -1$ (hence $m = [n/ \ell ]$). Then the $\ell$-core of the principal block is the partition $(a)$, and so the weight of this block must be $m$.

We are now ready to prove the main result in this section. It provides lower and upper bounds on the representation dimension of a block of a Hecke algebra of type $A$, in terms of the weight of the block.

\begin{theorem}\label{HeckeBlock}
Let $n \ge 2$ be an integer, and $q \in k^\times$. Furthermore, let $\mathcal{B}$ be a block of the Hecke algebra $\mathcal{H}_q(A_{n-1})$ of nonzero weight $w$, and assume that one of the following holds:
\begin{enumerate}
\item $k$ is algebraically closed of characteristic zero, and $q \in k \backslash \{0,1 \}$ has multiplicative order $\ell$ in $k^\times$, 
\item $k$ is algebraically closed of characteristic $p>w$, and $q \in \mathbb{F}_p \backslash \{0,1 \}$ has multiplicative order $\ell$ in $\mathbb{F}_p^\times$,
\item	$k$ is perfect of characteristic $p= \ell >w$ and $q=1$.
\end{enumerate}
Then
$$w+1 \le \repdim \mathcal{B} \le 2w.$$
\end{theorem}

\begin{proof}
If $\mathcal{B}$ is not a Rouquier block, then choose an integer $m$ with the property that the Hecke algebra $\mathcal{H}_q(A_{m-1})$ 
contains a Rouquier block $\mathcal{B'}$ of weight $w$ (such an integer always exists). 
By \cite[Theorem 7.12]{ChuangRouquier}, the algebras $\mathcal{B}$ and $\mathcal{B'}$ are derived equivalent. 
Consequently the block $\mathcal{B}$ is derived equivalent to a Rouquier block $\mathcal{B'}$ of weight $w$. 
Now let $\La$ be a Brauer tree algebra associated to the graph
$$\xymatrix{
\circ \ar@{-}[r] & \circ \ar@{-}[r] & \cdots \ar@{-}[r] & \circ \ar@{-}[r] & \circ }$$
with $\ell$ vertices and no exceptional multiplicity. 
By \cite[Theorem 18]{ChuangMiyachi} in cases (1) and (2), and \cite[Theorem 2]{ChuangKessar} in case (3), 
the algebra $\mathcal{B'}$ is Morita equivalent to the wreath product $\Lambda \wr S_w$. 
Let us make a note on the restrictions on $k$ and $q$ made in cases (1) and (2). These 
result from the fact that Chuang and Miyachi invoke the representation theory of the finite general linear group $GL_n(q)$ in their 
proof of \cite[Theorem 18]{ChuangMiyachi}. 
In case (2) we must assume that $q \in k \backslash \{0,1 \}$ can be obtained by reducing a prime power modulo $p$, 
and therefore lies in $\mathbb{F}_p$,
but we can take any such element since by Dirichlet's theorem any element of $\mathbb{F}_p$ is congruent modulo $p$ to some prime.
In cases (1) and (2) we must also assume the field $k$ contains many roots of unity since a splitting field for $GL_n(q)$ 
contains such; the assumption that $k$ is algebraically closed guarantees this.

The above shows that the block $\mathcal{B}$ is derived equivalent to the wreath product $\Lambda \wr S_w$. 
By \cite[Corollary 4.2]{Xi2}, the representation dimension of $\mathcal{B}$ equals that of $\Lambda \wr S_w$, 
hence we compute the bounds for the latter. 
Note also that in each case the field $k$ is perfect, and $w!$ is invertible in $k$.

The Brauer tree algebra $\La$ is selfinjective, not semisimple and of finite representation type. Therefore its representation dimension is two, and so the upper bound in Theorem \ref{WreathProduct} gives
$$\repdim \left ( \Lambda \wr S_w \right ) \le 2w.$$
To compute the lower bound, note that by \cite[Remark 4.10]{Linckelmann}, the algebra $\La$ is periodic as a bimodule. This implies that the Krull dimension of the Hochschild cohomology ring $\HH^*( \La )$ is one, and that the $\HH^*( \La )$-module $\Ext_{\Lae}( \La, X)$ is Noetherian for every $\La$-$\La$-bimodule $X$ (here $\Lae$ denotes the enveloping algebra $\La \otimes_k \La^{\op}$ of $\La$). Denote the tensor algebra $\La^{\otimes w}$ by $A$, and its Jacobson radical by $\ra$. By \cite[Proposition 4.8 and preceding paragraph]{Linckelmann}, the Krull dimension of  $\HH^*( A )$ is $w$, and the $\HH^*( A )$-module $\Ext_{A^{\e}}( A, Y)$ is Noetherian for every $A$-$A$-bimodule $Y$. By \cite[Proposition 2.4]{ErdmannEtAl}, the latter is equivalent to the following: the Hochschild cohomology ring $\HH^*( A )$ is Noetherian, and $\Ext_A^*(A/ \ra, A/ \ra)$ is a finitely generated $\HH^*( A )$-module. Since $A / \ra$ is separable, it follows from the proof of \cite[Corollary 3.6]{Bergh} that the complexity of the $A$-module $A / \ra$ equals the Krull dimension of $\HH^*( A )$, which is $w$. Then by \cite[Theorem 3.2]{Bergh}, there is an inequality $\dim \left ( \stmod A \right ) \ge w-1$, hence
$$w+1 \le \dim \left ( \stmod A \right ) +2.$$
Since $A = \La^{\otimes w}$, the lower bound
$$w+1 \le \repdim \left ( \Lambda \wr S_w \right )$$
now follows from Theorem \ref{WreathProduct}. This completes the proof.
\end{proof}

Part (3) of the theorem shows that for the group algebra $kS_n$ of the $n^{th}$ symmetric group, the bounds hold for blocks of weight less than the characteristic $p$ of $k$. In particular, if $n < p^2$, then the weight of any block of $kS_n$ is less than $p$, and so the bounds hold. We record this in the following corollary, which generalizes \cite[Theorem 4.4]{BerghErdmann}.

\begin{corollary}\label{SymBlock}
Let $k$ be a perfect field of positive characteristic $p$, and $n$ an integer with $p \le n$. Furthermore, let $\mathcal{B}$ be a block of $kS_n$ of nonzero weight $w < p$.
Then the inequalities
$$w+1 \le \repdim \mathcal{B} \le 2w$$
hold. In particular, if $n < p^2$, then the inequalities hold for every block of nonzero weight $w$.
\end{corollary}


\begin{thebibliography}{EHSST}
\bibitem[Au1]{Auslander1}M.\ Auslander, \emph{Representation
dimension of Artin algebras}, Queen Mary College Mathematics Notes,
London, 1971, republished in \cite{Auslander2}.
\bibitem[Au2]{Auslander2}M.\ Auslander, \emph{Selected works of
Maurice Auslander. Part 1}, I.\ Reiten, S.\ Smal{\o}, {\O}.\ Solberg
(editors), Amer.\ Math.\ Soc., Providence, 1999.
\bibitem[Ber]{Bergh}P.A.\ Bergh, \emph{Representation dimension and
finitely generated cohomology}, Adv.\ Math.\ 219 (2008), no.\ 1, 389-400.
\bibitem[BeE]{BerghErdmann}P.A.\ Bergh, K.\ Erdmann, \emph{The representation dimension of Hecke algebras and symmetric groups}, Adv.\ Math.\ 228 (2011), no.\ 4, 2503-2521.
\bibitem[BO1]{BerghOppermann1}P.A.\ Bergh, S.\ Oppermann, \emph{The
representation dimension of quantum complete intersections}, J.\
Algebra 320 (2008), no.\ 1, 354-368.
\bibitem[BO2]{BerghOppermann2}P.A.\ Bergh, S.\ Oppermann, \emph{Cohomology
of twisted tensor products}, J. Algebra 320 (2008), no.\ 8,
3327-3338.
\bibitem[ChK]{ChuangKessar}J.\ Chuang, R.\ Kessar, \emph{Symmetric groups, wreath products, Morita equivalences, and Brou{\'e}'s abelian defect group conjecture}, Bull.\ London Math.\ Soc.\ 34 (2002), no.\ 2, 174-185.
\bibitem[ChM]{ChuangMiyachi}J.\ Chuang, H.\ Miyachi, \emph{Runner removal Morita equivalences}, preprint.
\bibitem[ChR]{ChuangRouquier}J.\ Chuang, R.\ Rouquier, \emph{Derived equivalences for symmetric groups and $\mathfrak{sl}_2$-categorification}, Ann.\ of Math.\ (2) 167 (2008), no.\ 1, 245-298.
\bibitem[DiJ]{DipperJames}R.\ Dipper, G.D.\ James, \emph{Blocks and idempotents of Hecke algebras of general linear groups}, Proc.\ London Math.\ Soc.\ 54 (1987), no.\ 3, 57-82.
\bibitem[EHSST]{ErdmannEtAl}K.\ Erdmann, M.\ Holloway, N.\ Snashall,
  {\O}.\ Solberg, R.\ Taillefer, \emph{Support varieties for
    selfinjective algebras}, K-theory 33 (2004), no.\ 1, 67-87.
\bibitem[ErN]{ErdmannNakano}K.\ Erdmann, D.\ Nakano, \emph{Representation type of Hecke algebras of type $A$}, Trans.\ Amer.\ Math.\ Soc.\ 354 (2002), no.\ 1, 275-285.
\bibitem[Iya]{Iyama}O.\ Iyama, \emph{Finitness of representation
dimension}, Proc.\ Amer.\ Math.\ Soc.\ 131 (2003), no.\ 4,
1011-1014.
\bibitem[KrK]{KrauseKussin}H.\ Krause, D.\ Kussin, \emph{Rouquier's
theorem on representation dimension}, in \emph{Trends in
representation theory of algebras and related topics}, 95-103,
Contemp.\ Math.\ 406, Amer.\ Math.\ Soc., Providence, 2006.
\bibitem[Lin]{Linckelmann}M.\ Linckelmann, \emph{Finite generation of Hochschild cohomology of Hecke algebras of finite classical type in characteristic zero}, Bull.\ London Math.\ Soc.\ 43 (2011), no.\ 5, 871-885.
\bibitem[Op1]{Oppermann1}S.\ Oppermann, \emph{A lower bound for the
representation dimension of $kC_p^n$}, Math.\ Z.\ 256 (2007), no.\
3, 481-490.
\bibitem[Op2]{Oppermann2}S.\ Oppermann, \emph{Lower bounds for
Auslander's representation dimension}, Duke Math.\ J.\ 148 (2009), no.\ 2, 211-249.
\bibitem[Op3]{Oppermann3}S.\ Oppermann, \emph{Wild algebras have one-point extensions of representation dimension at least four}, J.\ Pure Appl.\ Algebra 213 (2009), no.\ 10, 1945-1960.
\bibitem[OpM]{OppermannMiemietz}S.\ Oppermann, V.\ Miemietz, \emph{On the representation dimension of Schur algebras}, Algebr.\ Represent.\ Theory 14 (2011), no.\ 2, 283-300.
\bibitem[ReR]{ReitenRiedtmann}I.\ Reiten, C.\ Riedtmann, \emph{Skew group algebras in the representation theory of Artin algebras}, J.\ Algebra 92 (1985), no.\ 1, 224-282. 
\bibitem[Ro1]{Rouquier1}R.\ Rouquier, \emph{Representation dimension
of exterior algebras}, Invent. Math.\ 165 (2006), no.\ 2, 357-367.
\bibitem[Ro2]{Rouquier2}R.\ Rouquier, \emph{Dimensions of
triangulated categories}, J.\ K-theory 1 (2008), no.\ 2, 193-256.
\bibitem[Xi1]{Xi}C.\ Xi, \emph{On the representation dimension of finite
dimensional algebras}, J.\ Algebra 226 (2000), no.\ 1, 332-346.
\bibitem[Xi2]{Xi2}C.\ Xi, \emph{Representation dimension and quasi-hereditary algebras}, Adv.\ Math.\ 168 (2002), no.\ 2, 193-212.
\end{thebibliography}
\end{document}